\documentclass{amsart}
\usepackage{amssymb}
\theoremstyle{plain}
\newtheorem{theorem}{Theorem}[section]
\newtheorem{proposition}[theorem]{Proposition}
\newtheorem{corollary}[theorem]{Corollary}
\theoremstyle{definition}
\newtheorem{definition}[theorem]{Definition}
\newtheorem*{example}{Example}
\newtheorem*{acks}{Acknowledgments}
\begin{document}
\title[IBN for $C^*$-Algebras]{Invariant Basis Number for $C^*$-Algebras}
\author{Philip M. Gipson}
\address{Department of Mathematics \\ State University of New York College at Cortland \\ Cortland, NY 13045-0900}
\email{philip.gipson@cortland.edu}
\begin{abstract}
We develop the ring-theoretic notion of Invariant Basis Number
in the context of unital $C^*$-algebras and their Hilbert
$C^*$-modules. Characterization of $C^*$-algebras with Invariant
Basis Number is given in $K$-theoretic terms, closure properties of
the class of $C^*$-algebras with Invariant Basis Number are given,
and examples of $C^*$-algebras both with and without the property
are explored. For $C^*$-algebras without Invariant Basis Number we
determine structure in terms of a ``Basis Type" and describe a class
of $C^*$-algebras which are universal in an appropriate sense. We
conclude by investigating properties which are strictly stronger
than Invariant Basis Number. 
\end{abstract}
\maketitle
\begin{section}{Introduction}
Leavitt \cite{leavitt3,leavitt4} investigated unital rings $R$ with the property that any free module $X$ over $R$ has a fixed basis size. Rings with this property are said to have Invariant Basis Number and examples of such include commutative and Noetherian rings. Leavitt characterizes \cite[Corollary 1]{leavitt4} rings with Invariant Basis Number in the following manner: a ring $R$ has Invariant Basis Number if and only if there exists another ring $R'$ with Invariant Basis Number and a unital homomorphism $\phi:R\to R'$. For rings without Invariant Basis Number, Leavitt assigns \cite[Theorem 1]{leavitt4} a pair of positive integers he terms the ``module type" of the ring. Constructions \cite{leavitt1,leavitt3,leavitt4} of rings, termed Leavitt algebras $L_K(m,n)$, with arbitrary module type are given.

The fundamental structure of the Leavitt algebras has appeared in some surprising contexts. The algebra $L_K(1,n)$ given by Leavitt \cite[\S 3]{leavitt4} is the purely algebraic analogue of the Cuntz $C^*$-algebra $\mathcal{O}_n$ and pre-dates Cuntz's investigations. Indeed, the close connection between Leavitt algebras and Cuntz algebras inspired the formulation of Leavitt Path Algebras associated to graphs, which act as analogues to graph $C^*$-algebras. General Leavitt algebras $L_K(m,n)$ have been investigated by Ara and Goodearl \cite{ara} in the context of ``separated" Leavitt Path Algebras. Several $C^*$-algebraic versions of the Leavitt algebras $L_K(m,n)$ have been recently used in the work of Exel and Ara \cite{exel2,exel} related to dynamical systems. 

In this paper we will formulate the property of Invariant Basis Number in the context of $C^*$-algebras and their Hilbert $C^*$-modules. Using $K$-theoretic tools we are able to formulate an improved characterization of $C^*$-algebras with Invariant Basis Number in Theorem \ref{ibnchar}. We reproduce in Theorem \ref{basistype} Leavitt's type-classification for $C^*$-algebras without Invariant Basis Number and prove in Theorem \ref{typeexists} that each Basis Type is possible for some $C^*$-algebra. In Section \ref{univ} we determine that the $C^*$-algebras $U_{m,n}^{nc}$ studied by McCLanahan \cite{mcclanahan} are universal objects for $C^*$-algebras without Invariant Basis Number and, as such, are the correct analogue of the Leavitt algebras $L_K(m,n)$. Finally we will investigate several stronger variations of Invariant Basis Number as proposed in the purely algebraic case by Cohn \cite{cohn}.

\end{section}
\begin{section}{$C^*$-Module Preliminaries}
We will always assume our $C^*$-algebras to be unital and denote the unit by by $1$ or $1_A$. A \textit{$C^*$-module} $X$ over a $C^*$-algebra $A$ (more briefly, an \textit{$A$-module}) is a complex vector space which is a right $A$-module and is equipped with an $A$-valued inner-product $\langle\cdot,\cdot\rangle:X\times X\to A$ which is $A$-linear in the second coordinate and $A$-adjoint-linear in the first coordinate. If $X$ is complete with respect to the norm $||x||=||\langle x,x\rangle||_A^{\frac12}$ then it is termed a \textit{Hilbert $A$-module}. We will use Wegge-Olsen \cite[Chapter 15]{wegge} as a reference for basic Hilbert $C^*$-module results.

The space of adjointable $A$-module homomorphisms between two $A$-modules $X$ and $Y$ will be denoted $L(X,Y)$. An adjointable homomorphism $\phi$ is \textit{unitary} if it is bijective and isometric, i.e.\ $\langle x,x'\rangle_X=\langle \phi(x),\phi(x')\rangle _Y$ for all $x,x'\in X$. We will say that $X$ and $Y$ are \textit{unitarily equivalent}, and write $X\simeq Y$, if there exists a unitary in $L(X,Y)$.

An $A$-module $X$ is \textit{algebraically finitely generated} if there exist $x_1,...,x_n\in X$ such that $X=\operatorname{span}_A(x_1,...,x_n)$. We will never consider the weaker notion of topological finite generation, and so will omit the term ``algebraically" in the remainder. An $A$-module $X$ is \textit{projective} if it is a direct summand of a free $A$-module. It is a known result (\cite[Theorem 15.4.2]{wegge} for example) that a finitely generated projective $A$-module is isomorphic (as an $A$-module) to a Hilbert $A$-module. Further, the finitely generated projective Hilbert $A$-modules are all of the form $pA^n$ for some $n\geq 1$ and some matrix projection $p\in M_n(A)$.

We will denote the set of projections in $M_n(A)$ by $P_n(A)$. For $p\in P_n(A)$ and $q\in P_m(A)$ we will set $p\oplus q=diag(p,q)\in P_{n+m}(A)$. We will say $p$ and $q$ are \emph{stably equivalent} if there is a matrix projection $r$ for which $p\oplus r\sim q\oplus r$, where ``$\sim$" denotes (Murray-von Neumann) equivalence in $P_\infty(A)=\bigcup_{n=1}^\infty P_n(A)$. The stable equivalence class of $p$ will be denoted $[p]_0$ and considered as an element of the group $K_0(A)$. The (additive) order of an element $[p]_0\in K_0(A)$ will be denoted $|[p]_0|_{K_0(A)}$ or $|[p]_0|$ if the $C^*$-algebra $A$ is clear from context.

\end{section}

\begin{section}{Invariant Basis Number}\label{IBN}

Let $A$ be a unital $C^*$-algebra. The \emph{finitely generated free $A$-module of rank $n$} is $A^n:=A\oplus...\oplus A$ where there are $n$ summands. The action of $A$ on $A^n$ is coordinate-wise multiplication on the right and the inner-product is given by $\langle (a_1,...,a_n),(b_1,...,b_n)\rangle=a_1^*b_1+...+a_n^*b_n$. Although we write them as tuples, i.e.\ row vectors, it is often beneficial to view elements of $A^n$ instead as column vectors. The coordinate projections  $\pi_i:A^n\to A$  defined by $\pi_i(a_1,...,a_n)=a_i$ are  bounded, contractive, adjointable $A$-module homomorphisms. Therefore a Cauchy sequence in $A^n$ is Cauchy in each coordinate and hence, as $A$ itself is complete, converges in each coordinate. Thus $A^n$ is a complete (i.e.\ Hilbert) $A$-module. In keeping with the literature, free Hilbert $A$-modules will henceforth be referred to as \emph{standard $A$-modules}, where the completeness is understood.

The fundamental question we will consider is this: under what conditions are the standard modules distinct from one another? We will make this notion of distinctness precise with the next definition.

\begin{definition} A $C^*$-algebra $A$ has \emph{Invariant Basis Number} (hereafter, has IBN) if
$$A^n\simeq A^m\Leftrightarrow n=m.$$
\end{definition}

Unitary equivalence is, in general, a stronger condition than $A$-module isomorphism. In fact, unitaries are precisely the \emph{isometric} $A$-module isomorphisms. However, in the case of standard modules every $A$-module homomorphism $\phi:A^n\to A^m$ may be represented as a $m\times n$ matrix with elements in $A$ and so is automatically adjointable. Therefore if $\phi:A^n\to A^m$ is an $A$-module isomorphism then the Polar Decomposition \cite[Theorem 15.3.7]{wegge} yields a unitary in $L(A^n,A^m)$. We have formulated the definition in terms of unitary equivalence, rather than module isomorphism, to emphasize the Hilbert structure of the standard modules.

A matrix $U\in M_{m,n}(A)$ will be termed a \emph{unitary} if $UU^*=I_n$ and $U^*U=I_m$. As noted above, we may identify $L(A^n,A^m)$ with $M_{m,n}(A)$ and a unitary homomorphism in $L(A^n,A^m)$ corresponds to a unitary matrix in $M_{m,n}(A)$. The definition of Invariant Basis Number may thus be rephrased as follows: $A$ has IBN if and only if every unitary matrix over $A$ is square.

\begin{example} It is not hard to verify that a matrix with entries in an commutative algebra is invertible if and only if it is square. Hence commutative $C^*$-algebras have Invariant Basis Number.
\end{example}

The connection between matrices and Invariant Basis Number gives our first main result.

\begin{theorem}\label{ibnchar} A $C^*$-algebra $A$ has \textup{IBN} if and only if the group element $[1_A]_0\in K_0(A)$ has infinite order.
\end{theorem}
\begin{proof} If $A$ does not have IBN then $A^n\simeq A^m$ for some $n>m>0$ and hence there is a unitary matrix in $M_{m,n}(A)$. This unitary implements the (Murray-von Neumann) matrix equivalence of the projections $I_m$ and $I_n$ and consequently we have $$I_{n-m}\oplus I_m\sim I_n\sim I_m\sim0\oplus I_m.$$
Thus $I_{n-m}$ is stably equivalent to $0$, i.e.\ $(n-m)[1_A]_0=[I_{n-m}]_0=0$, and so $[1_A]_0$ has finite order.

Conversely, if $[1_A]_0$ has finite order $k$ then $I_k$ is stably equivalent to $0$, i.e.\ there exists $N>0$ and $p\in P_N(A)$ such that $p\oplus I_k\sim p\oplus 0\sim p$. As $I_N\sim p\oplus (I_N-p)$ we have
$$I_N\oplus I_k\sim(I_N-p)\oplus p\oplus I_k\sim (I_N-p)\oplus p\sim I_N$$
and so $I_{N+k}\sim I_N$.
The matrix implementing this equivalence is unitary and thus corresponds to a unitary homomorphism from $A^N$ to $A^{N+k}$. Since $k>0$ we must conclude that $A$ does not have IBN.
\end{proof}

It is hinted in the above proof that when a $C^*$-algebra does not have IBN the order of $[1_A]_0$ yields information about equivalence of standard modules. We shall make this connection clear in Section \ref{nonIBN} when we turn our attention fully to $C^*$-algebras without IBN.

The $K$-theoretic description of IBN immediately expands the class of $C^*$-algebras with that property beyond the commutative. In particular, it is well-known (see \cite{rordambook}, for example) that stably-finite $C^*$-algebras, i.e.\ those without any proper matrix isometries, have a totally ordered $K_0$ group. Further, in this case the element $[1_A]_0$ is an \emph{order unit} for $K_0$ in the sense that for any $g\in K_0$ there is a positive integer $k$ for which $-k[1_A]_0<g<k[1_A]$. It follows that $[1_A]_0$ cannot have a finite order and, applying Theorem \ref{ibnchar}, we conclude that a stably-finite $C^*$-algebra must have IBN. We would like to remark that this could also be inferred from the matricial description of IBN, as any rectangular unitary could be ``cut down" to a square proper isometry.

The funtorial properties of $K_0$ also yield the following result which will be used extensively to demonstrate closure properties for the class of $C^*$-algebras with IBN.

\begin{proposition}\label{prophomo} A $C^*$-algebra $A$ has IBN if and only if there exists a $C^*$-algebra $B$ which has IBN and a unital $*$-homomorphism $\phi:A\to B$.
\end{proposition}
\begin{proof} Necessity is easily satisfied by letting $B=A$ and $\phi=id_A$.

To show sufficiency,  we note that the functorial properties of $K_0$ induce a group homomorphism $K_0(\phi):K_0(A)\to K_0(B)$. Since $\phi$ is unital we have $K_0(\phi)[1_A]_0=[1_B]_0$. If $B$ has IBN then $[1_B]_0$ has infinite order in $K_0(B)$ and so its preimage $[1_A]_0$ must have infinite order in $K_0(A)$. Thus $A$ has IBN.
\end{proof}

The above statement mirrors the purely algebraic characterization of rings with IBN given by Leavitt \cite[Corollary 1]{leavitt4}.

The proposition has immediate consequences for the closure properties of the class of $C^*$-algebras with Invariant Basis Number.

\begin{corollary}\label{closures} \textup{IBN} is preserved under direct sums and unital extensions.
\end{corollary}
\begin{proof}
Suppose that $A$ is a $C^*$-algebra with IBN. If $B$ is a unital $C^*$-algebra then the coordinate map $a\oplus b\mapsto a$ is a unital $*$-homomorphism and thus $A\oplus B$ has IBN.

If $B$ is any unital extension of $A$ then there exists a $C^*$-algebra $C$ and a short exact sequence
$$0\rightarrow C\rightarrow B\xrightarrow{\phi} A\rightarrow 0.$$
Of course $\phi$ is a surjective $*$-homomorphism, hence is unital, and thus $B$ has IBN.
\end{proof}

Note that a direct sum inherits IBN even if only one of the summands has that property. We conclude our discussion of $C^*$-algebras with IBN by leveraging the results to find non-commutative, non-stably-finite $C^*$-algebras which have IBN.

\begin{example} Consider the Cuntz algebra $\mathcal{O}_\infty$, the universal $C^*$-algebra generated by a countable family of isometries with pairwise disjoint ranges. Since $\mathcal{O}_\infty$ contains proper isometries it is certainly neither commutative nor (stably) finite. However, it is a classical result of Cuntz \cite[Corollary 3.11]{cuntzk} that $K_0(\mathcal{O}_\infty)=\mathbb{Z}$ and is generated by $[1]_0$. Thus by, Theorem \ref{ibnchar}, $\mathcal{O}_\infty$ has IBN.\end{example}

\begin{example} On the opposite end of the spectrum, consider the Toeplitz algebra $\mathcal{T}$, the universal $C^*$-algebra generated by a single non-unitary isometry. Of course $\mathcal{T}$ is neither commutative nor (stably) finite but is well known to be an extension of the commutative $C^*$-algebra $C(\mathbb{T})$ by the compact operators $\mathcal{K}$. Thus by Corollary \ref{closures} $\mathcal{T}$ has IBN.\end{example}

\begin{subsection}{A remark on the non-unital case.}
It is a perfectly legitimate criticism that we are dealing solely with unital $C^*$-algebras. Let us briefly describe why we wish to avoid the nonunital case. 

Suppose that $A$ is a nonunital $C^*$-algebra. Unlike in the unital case, the adjointable $A$-module homomorphisms in $L(A^n,A^m)$ are not identified with $M_{m,n}(A)$, but rather with $m\times n$ matrices over the \emph{multiplier algebra} of $A$, which we'll denote by $\mathcal{M}(A)$. Of course $\mathcal{M}(A)$ is, practically by definition, unital. The unitary equivalence $A^n\simeq A^m$ thus implies the existence of a unitary matrix in $M_{m,n}(\mathcal{M}(A))$ and so $\mathcal{M}(A)^n\simeq\mathcal{M}(A)^m$. It is not hard to see that the logic is reversible and so $A^n\simeq A^m$ if and only if $\mathcal{M}(A)^n\simeq\mathcal{M}(A)^m$. 

As a consequence of the above reasoning, we see that the statement ``$A^n\simeq A^m$ if and only if $n=m$" is equivalent to ``$\mathcal{M}(A)$ has IBN." This is what we believe should be the working definition of IBN for nonunital $C^*$-algebras. In fact, since $\mathcal{M}(A)=A$ when $A$ is unital, it agrees with our unital definition.

Unfortunately, we do not feel this definition to be particularly useful. First, many nice properties of a $C^*$-algebra are not preserved in it's multiplier algebra. Seperability being a prime example. Second, we do not know of a method, outside a very few special cases, to detect information about $K_0(\mathcal{M}(A))$ based on information about $A$. Since our main tools are $K$-theoretic this is a major stumbling block.

\end{subsection}
\end{section}

\begin{section}{$C^*$-Algebras Without Invariant Basis Number}\label{nonIBN}
We now turn our attention to those unital $C^*$-algebras which lack the Invariant Basis Number property. By Theorem \ref{ibnchar}, we may conclude that $C^*$-algebras $A$ without IBN are characterized by having a finite order for the element $[1_A]_0\in K_0(A)$. A particularly tractable case is when $[1_A]_0$ has order 1, i.e.\ is the zero element of $K_0(A)$.
\begin{example} When $H$ is an infinite dimensional Hilbert space $B(H)$ does not have IBN because $K_0(B(H))=\{0\}$. \end{example}

\begin{example} The Cuntz algebra $\mathcal{O}_2$ is the universal $C^*$-algebra generated by two isometries $v_1$ and $v_2$ satisfying $v_1v_1^*+v_2v_2^*=1$ and $v_1^*v_2=v_2^*v_1=0$. A result of Cuntz \cite[Theorem 3.7]{cuntzk} is that $K_0(\mathcal{O}_2)=\{0\}$ and so $\mathcal{O}_2$ does not have IBN. In fact we can concretely see the equivalence $\mathcal{O}_2\simeq\mathcal{O}_2^2$ via the map $(a,b)\mapsto v_1a+v_2b$ which extends to a unitary homomorphism and corresponds to the $1\times 2$ unitary matrix $[v_1\ v_2]$. \end{example}

\begin{example} For a slightly less trivial example, consider the Cuntz algebra $\mathcal{O}_3$. We have that $K_0(\mathcal{O}_3)=\mathbb{Z}/2\mathbb{Z}$ and is in fact generated by $[1]_0$. Thus $\mathcal{O}_3$ does not have IBN. Much like for $\mathcal{O}_2$ we can in fact write down a $1\times 3$ unitary matrix $[v_1\ v_2\ v_3]$ which gives the unitary equivalence $\mathcal{O}_3\simeq\mathcal{O}_3^3$. Of course in general we have $K_0(\mathcal{O}_n)=\mathbb{Z}/(n-1)\mathbb{Z}$ and so no Cuntz algebra has IBN.\end{example}

Recalling the definition of Invariant Basis Number, a $C^*$-algebra lacks IBN precisely when two or more standard modules with differing ranks are equivalent. The restrictions on when such equivalence may occur give some structural information for $C^*$-algebras without IBN. The precise nature of that information is contained in our next main result.
\begin{theorem}\label{basistype} If $A$ is a $C^*$-algebra without \textup{IBN} then there exists a unique largest positive integer $N$ and a unique smallest positive integer $K$ satisfying:
\begin{enumerate}
\item if $n,m\geq 1$, $n<N$, and $A^n\simeq A^m$ then $n=m$, and
\item if $n,m\geq 1$ and $A^n\simeq A^m$ then $(n-m)\equiv 0\mod K$.
\end{enumerate}
\end{theorem}
This result is comparable to \cite[Theorem 1]{leavitt4}. The first condition characterizes $N$ as the least rank for which distinctness of the standard $A$-modules fails: all standard $A$-modules of rank less than $N$ are distinct. The second condition characterizes $K$ as the minimum ``jump" in rank possible between equivalent standard $A$-modules.
\begin{definition} If $A$ is a $C^*$-algebra without IBN then the pair $(N,K)$ given by Theorem \ref{basistype} is the \emph{Basis Type} of $A$. For notational purposes we may write $type(A)=(N,K)$ or $(N_A,K_A)$.
\end{definition}
\begin{proof}[Proof of Theorem \ref{basistype}.] Since $A$ does not have IBN there are at least two distinct ranks $n,m$ for which $A^n\simeq A^m$. In particular, the set $E:=\{j\geq0:\exists k\not=j\ s.t.\ A^j\simeq A^k\}$ is nonempty and so $N:=\min\{n:n\in E\}$ is well defined. If $n<N$ then $n\not\in E$ and so $A^n\simeq A^m$ only if $m=n$. So our choice of $N$ satisfies the first condition. That our $N$ is the largest possible is immediate, since if $N'>N$ then there is at least one rank ($N$ itself) less than $N'$ for which the first condition does not hold.

Let $N$ be as above and define $K=\min\{k>0:A^N\simeq A^{N+k}\}$, which exists by our choice of $N$. Note that for any $n\geq N+K$ we have
$$A^n=A^{n-N-K+N+K}\simeq A^{n-N-K}\oplus A^{N+K}\simeq A^{n-N-K}\oplus A^N\simeq A^{n-K}.$$
Through iteration of this process we obtain an integer $n'$ satisfying $N\leq n'<N+K$, $n'\equiv n\mod K$, and $A^{n'}\simeq A^n$. Because of this, it is enough to check a simpler version of the second condition: if $A^n\simeq A^m$ for $N\leq n,m<N+K$ then $n=m$. (Note this will guarantee the minimality of $K$.)
Suppose that $n,m$ are two ranks satisfying the simplified hypothesis but with $m>n$. Then
$$A^N\simeq A^{N+K}\simeq A^{N+K-m}\oplus A^m\simeq A^{N+K-m}\oplus A^n\simeq A^{N+K-(m-n)}$$
and, as $K-(m-n)<K$, we have contradicted the minimality of $K$.
\end{proof}

The Basis Type of a $C^*$-algebra determines the equivalences of standard modules. In particular, if $type(A)=(N,K)$ then there are precisely $N+K$ unitary equivalence classes of standard modules: the distinct ones of rank less than $N$ and the $K$ classes for ranks $N,\ N+1, ...,\ N+K-1$.

\begin{example} Revisiting the examples from the beginning of the section, we find that $B(H)$ and $\mathcal{O}_2$ both have Basis Type $(1,1)$. The Cuntz algebra $\mathcal{O}_3$ is of Basis Type $(1,2)$ since (as may be checked) $\mathcal{O}_3\not\simeq \mathcal{O}_3^2$ but $\mathcal{O}_3\simeq\mathcal{O}_3^3$.\end{example}
Recalling that $K_0(\mathcal{O}_2)=K_0(B(H))=0$ while $K_0(\mathcal{O}_3)=\mathbb{Z}/2\mathbb{Z}$ the following proposition is perhaps unsurprising.

\begin{proposition}\label{kisk} If $A$ is a $C^*$-algebra with Basis Type $(N,K)$ then the order of $[1_A]_0$ in $K_0(A)$ is equal to $K$.
\end{proposition}
\begin{proof}
Since $A$ does not have IBN the element $[1_A]_0$ must have some finite order $J$.  Since $A^{N}\simeq A^{N+K}$ by definition of the Basis Type we conclude that $I_N$ and $I_{N+K}$ are (Murray-von Neumann) equivalent matrix projections; consequently we have $K[1_A]_0=[I_K]_0=0$ in $K_0(A)$ and thus $K\equiv 0\mod J$. Re-examination of the proof for Theorem \ref{ibnchar} yields that as $J[1_A]_0=0$ there exists some $M$ such that $I_{M+J}\sim I_M$, i.e.\ $A^{M}\simeq A^{M+J}$. Thus, by definition of $K$, we have $J\equiv 0\mod K$. We must then conclude that $J=K$, as desired.
\end{proof}

Following Leavitt \cite[\S2]{leavitt4}, we will give the Basis Types a lattice structure as follows:
$$(N_1,K_1)\leq (N_2,K_2)\Leftrightarrow N_1\leq N_2\ and\ K_2\equiv 0\mod K_1,$$
$$(N_1,K_1)\vee (N_2,K_2)=(\max(N_1,N_2),\operatorname{lcm}(K_1,K_2)),$$
$$(N_1,K_1)\wedge (N_2,K_2)=(\min(N_1,N_2),\operatorname{gcd}(K_1,K_2).$$

We are able to relate this lattice structure to various algebraic operations primarily through the following proposition.

\begin{proposition}\label{typedown} Let $A$ and $B$ be $C^*$-algebras, $A$ without \textup{IBN}, and $\phi:A\to B$ a unital $*$-homomorphism. Then $B$ is without \textup{IBN} and $type(B)\leq type(A)$.
\end{proposition}
\begin{proof} Note that by Proposition \ref{prophomo} $B$ cannot have IBN. Let $type(A)=(N_A,K_A)$ and $type(B)=(N_B,K_B)$. The functoriality of $K_0$ induces a group homomorphism $K_0(\phi):K_0(A)\to K_0(B)$ which takes $[1_A]_0$ to $[1_B]_0$. Being a group homomorphism, it follows that the order of $K_0(\phi)[1_A]_0\in K_0(B)$ must divide the order of $[1_A]_0\in K_0(A)$. We thus have
$$\left|[1_A]_0\right|_{K_0(A)}\equiv 0\mod\left|[1_B]_0\right|_{K_0(B)}$$
which combines with Proposition \ref{kisk} to give us $K_A\equiv 0\mod K_B$.

We may ampliate $\phi$ to $\phi^{(m,n)}:M_{m,n}(A)\to M_{m,n}(B)$ by applying $\phi$ entry-wise. Since $\phi$ is unital any unitary matrix in $M_{m,n}(A)$ is sent, via $\phi^{(m,n)}$, to a unitary matrix in $M_{m,n}(B)$. Thus if $A^n\simeq A^m$ then so too $B^n\simeq B^m$; in particular we have $B^{N_A}\simeq B^{N_A+K_A}$. By construction (see Theorem \ref{basistype}) $N_B=\min\{n:\exists j\not=n\ s.t.\ B^n\simeq B^{j}\}$ and so we conclude that $N_B\leq N_A$.
\end{proof}
The primary utility of the previous proposition is to prove various closure properties of the class of $C^*$-algebras without IBN.

\begin{corollary} If $A$ does not have \textup{IBN} and $B$ is a quotient of $A$ then $B$ does not have \textup{IBN}.
\end{corollary}
This is Proposition \ref{typedown} applied to the quotient map.

\begin{corollary}\label{sums} If $A$ and $B$ are $C^*$-algebras without \textup{IBN} then $type(A\oplus B)=type(A)\vee type(B)$.
\end{corollary}
\begin{proof} Proposition \ref{typedown} applied to the coordinate projections $(a,b)\mapsto a$ and $(a,b)\mapsto b$ has us conclude that $type(A)\leq type(A\oplus B)$ and $type(B)\leq type(A\oplus B)$ and so $type(A)\vee type(B)\leq type(A\oplus B)$.

As $K_0(A\oplus B)=K_0(A)\oplus K_0(B)$ we use Proposition \ref{kisk} to conclude that $K_{A\oplus B}=\operatorname{lcm}(K_A,K_B)$.

Suppose, without loss of generality, that $\max(N_A,N_B)=N_A$. With $K_{A\oplus B}=\operatorname{lcm}(K_A,K_B)$, we have
$$A^{N_A}\simeq A^{N_A+K_A}\simeq A^{N_A+2K_A}\simeq...\simeq A^{N_A+K_{A\oplus B}}$$
and, as $B^{N_A}\simeq B^{N_A-N_B}\oplus B^{N_B}\simeq B^{N_A-N_B}\oplus B^{N_B+K_B}\simeq B^{N_A+K_B}$, we have also
$$B^{N_A}\simeq B^{N_A+K_B}\simeq B^{N_A+2K_B}\simeq...\simeq B^{N_A+K_{A\oplus B}}.$$
Consequently
$$(A\oplus B)^{N_A}=A^{N_A}\oplus B^{N_A}\simeq A^{N_A+K_{A\oplus B}}\oplus B^{N_A+K_{A\oplus B}}\simeq (A\oplus B)^{N_A+K_{A\oplus B}}.$$
We conclude that $N_{A\oplus B}\leq N_A=\max(N_A,N_B)$. As $type(A)\wedge type(B)\leq type(A\oplus B)$, i.e.\ $\max(N_A, N_B)\leq N_{A\oplus B}$, we have equality.

In conclusion $N_{A\oplus B}=\max(N_A,N_B)$ and $K_{A\oplus B}=\operatorname{lcm}(K_A,K_B)$ and so $type(A\oplus B)=type(A)\vee type (B)$.
\end{proof}

In contrast to Corollary \ref{closures} it is quite necessary that neither summand of $A\oplus B$ has IBN. It is natural to suspect that the remaining lattice operation will correspond to tensor products.

\begin{corollary} If $A$ and $B$ are $C^*$-algebras without \textup{IBN} then $type(A\otimes B)\leq type(A)\wedge type(B)$.
\end{corollary}
The proof of this corollary is nothing but Proposition \ref{typedown} applied to the embeddings $a\mapsto a\otimes 1_B$ and $b\mapsto 1_A\otimes b$. Two remarks are in order: first, that the result holds for any norm structure we may place on $A\otimes B$; second, that it is unknown (even, to our knowledge, in the purely algebraic case) whether inequality ever occurs.

\begin{corollary} If $\{A_i,\phi_{i}\}$ is an inductive system of $C^*$-algebras, each without \textup{IBN}, and each $\phi_i$ is unital, then the direct limit $C^*$-algebra $A$ of the system does not have IBN.
\end{corollary}
The proof of this corollary is Proposition \ref{typedown} applied to the universal maps $\psi_i:A_i\to A$, which are unital.

Finally, we will demonstrate that the class of $C^*$-algebras without Invariant Basis Number is unfortunately \emph{not} closed under Morita equivalence. A good reference for the theory of Morita equivalence is \cite{raeburn}. Our motivating example is the algebra $\mathcal{O}_\infty$ and the fact that the identity of a corner $C^*$-algebra $p\mathcal{O}_\infty p$ is the projection $p$.

\begin{proposition}\label{morita} Let $A$ be a infinite simple unital $C^*$-algebra, then there is a $C^*$-algebra $B$ Morita equivalent to $A$ which does not have \textup{IBN}.
\end{proposition}
\begin{proof} If $A$ is infinite then there exists a proper isometry $v\in A$. As $vv^*\sim v^*v=1_A$ we have
$$[1_A]_0=[1_A-vv^*]_0+[vv^*]_0=[1_A-vv^*]_0+[1_A]_0$$
and so $[1_A-vv^*]_0=0$ in $K_0(A)$. Now consider the full corner $B=(1_A-vv^*)A(1_A-vv^*)$, which is Morita-equivalent to $A$ \cite[Example 3.6]{raeburn}, and note that $1_B=1_A-vv^*$. Thus $[1_B]_0=0$ in $K_0(B)$ and so $B$ does not have IBN.
\end{proof}

Returning to the concrete example, $\mathcal{O}_\infty$ is a unital simple infinite $C^*$-algebra. We have seen before that $\mathcal{O}_\infty$ has IBN but now, by the above Proposition, it contains many full corners which does not have IBN.

\end{section}

\begin{section}{Universal Algebras for Basis Types}\label{univ}
A natural question stemming from the discussion of Basis Type is this: are all pairs $(N,K)$ of positive integers realized as the Basis Types of $C^*$-algebras? We shall answer this in the affirmative and further we will exhibit $C^*$-algebras which are ``universal" for their Basis Type.

Our investigation will be motivated by the situation for the Basis Types $(1,K)$. If $type(A)=(1,K)$ then necessarily $A\simeq A^{K+1}$ and so there is a unitary $1\times (K+1)$ matrix, i.e.\ a row unitary. The elements of such a matrix are isometries satisfying the Cuntz relations and so there is an induced unital $*$-homomorphism (in fact, an embedding) of $\mathcal{O}_{K+1}$ into $A$. Now as $\mathcal{O}_{K+1}\simeq\mathcal{O}_{K+1}^{K+1}$ and $K_0(\mathcal{O}_{K+1})=\mathbb{Z}/K\mathbb{Z}$ we conclude via Proposition \ref{kisk} that $type(\mathcal{O}_{K+1})=(1,K)$. We consider the Cuntz algebra $\mathcal{O}_{K+1}$ ``universal" for Basis Type $(1,K)$ in this sense: whenever $type(A)=(1,K)$ there is an induced unital $*$-homomorphism $\phi:\mathcal{O}_{K+1}\to A$. We use the term universal loosely because this homomorphism is not necessarily unique. For example, when $A$ is itself a Cuntz algebra then $\phi$ can be given by any permutation of the generating isometries.

In \cite{mcclanahan} McClanahan investigated $C^*$-algebras $U_{m,n}^{nc}$ defined as follows:
$$U_{m,n}^{nc}:=C^*\left(u_{ij}:U=[u_{ij}]\in M_{m,n}\ \textrm{satisfies}\ UU^*=I_m,\ U^*U=I_n\right).$$
The $C^*$-algebra $U_{m,n}^{nc}$ has the universal property that whenever $A$ is a $C^*$-algebra with elements $\{a_{ij}\}$ such that $[a_{ij}]\in M_{m,n}(A)$ is unitary then there is a unital $*$-homomorphism $\phi:U_{m,n}^{nc}\to A$ with $\phi(u_{ij})=a_{ij}$. Since there is a natural identification of $U_{m,n}^{nc}$ with $U_{n,m}^{nc}$ (taking $u_{ij}$ to $u_{ji}^*$) we shall only consider the cases where $n>m$.

Suppose that $A$ is a $C^*$-algebra with $type(A)=(N,K)$. Then by definition $A^N\simeq A^{N+K}$ and so there is an $N\times (N+K)$ unitary matrix over $A$. By the universal property we have a unital $*$-homomorphism $\phi:U_{N,N+K}^{nc}\to A$. Thus we may recast the universal property enjoyed by the $U_{m,n}^{nc}$ as follows: if $A$ is a $C^*$-algebra of Basis Type $(m,n-m)$ then there is a unital $*$-homomorphism $\phi:U_{m,n}^{nc}\to A$. McClanahan proved that $U_{1,n}^{nc}=\mathcal{O}_n$ and so there is no conflict with our previous discussion. He further demonstrated that $U_{m,n}^{nc}$ is not simple whenever $m>0$ (there is always a unital $*$-homomorphism $\phi:U_{m,n}^{nc}\to \mathcal{O}_{n-m+1}$) and so, unlike for the Cuntz algebras, the universal property does not guarantee an embedding of $U_{m,n}^{nc}$ into a $C^*$-algebra when $m>1$.

Since $U_{m,n}^{nc}$, by definition, has a unitary $m\times n$ matrix we conclude that its standard modules of ranks $n$ and $m$ are equivalent, and so $U_{m,n}^{nc}$ does not have IBN. Ara and Goodearl have recently shown in \cite{ara} that $K_0(U_{m,n}^{nc})=\mathbb{Z}/(n-m)\mathbb{Z}$ (and is generated by $[1]_0$) and so by Proposition \ref{kisk} we have that $type(U_{m,n}^{nc})=(N,n-m)$ for some $N\leq m$. To prove that we have $N=m$ we shall exploit the universal property of $U_{m,n}^{nc}$ together with our next main result.

\begin{theorem}\label{typeexists} For each pair $(N,K)$ of positive integers there is a $C^*$-algebra $A$ with $type(A)=(N,K)$.
\end{theorem}
\begin{proof}
We have already seen that for $K>0$, $type(\mathcal{O}_{K+1})=(1,K)$. As $(1,K) \vee (N,1)=(N,K)$ we conclude by Corollary \ref{sums} that it is enough, given $N>0$, to exhibit a $C^*$-algebra of Basis Type $(N,1)$.

By combining \cite[Theorem 3.5]{rordamsums} and \cite[Theorem 5.3]{rordamstability} we may, for fixed $N>0$, obtain a unital $C^*$-algebra $A$ with the following properties:
\begin{enumerate}
\item for $n<N$ the $C^*$-algebras $M_n(A)$ are finite,
\item for $m\geq N$ the $C^*$-algebras $M_m(A)$ are properly infinite, and
\item $K_0(A)=0$.
\end{enumerate}
Since $K_0(A)=0$ it follows that from Theorem \ref{ibnchar} and Proposition \ref{kisk} that $A$ does not have IBN and has basis type $(N',1)$ for some $N'>0$.
Since $K_0(M_N(A))=K_0(A)=0$ and $M_N(A)$ is properly infinite there is an embedding (see  \cite[Prop. 4.2.3]{rordamclass}) of $\mathcal{O}_2$ into $M_N(A)$. Thus there is a $1\times 2$ unitary matrix (with entries consisting of the images of the Cuntz isometries) over $M_N(A)$ which, viewed in a different light, is an $N\times 2N$ unitary matrix over $A$ itself. Thus $A^N\simeq A^{2N}$ and we conclude that $N'\leq N$. Suppose that $N'<N$. As $type(A)=(N',1)$ we have $A^{N'}\simeq A^{N'+1}$ and so there is a unitary $N'\times (N'+1)$ matrix. Deleting any one column from this matrix yields a $N'\times N'$ proper isometry, contradicting the fact that $M_{N'}(A)$ is finite. Hence $N'=N$ and $type(A)=(N,1)$.
\end{proof}
We emphasize that the $C^*$-algebras in Theorem \ref{typeexists} (obtained from \cite{rordamstability} and \cite{rordamsums}) are not simple. Since the $C^*$-algebras $U_{m,n}^{nc}$ are also not simple in general, it is a question of some interest to us if Basis Types beyond $(1,K)$ are possible for simple $C^*$-algebras.
\begin{corollary} $type(U_{m,n}^{nc})=(m,n-m).$
\end{corollary}

This is obtained from Theorem \ref{typeexists}, Proposition \ref{typedown}, and the universal property of $U_{m,n}^{nc}$. 

\begin{corollary} $U_{m,n}^{nc}=U_{m',n'}^{nc}$ if and only if $n=n'$ and $m=m'$.
\end{corollary}

Note that the Basis Types are able to distinguish the $C^*$-algebras $U_{m,n}^{nc}$ and $U_{m+1,n+1}^{nc}$ while the $K$-theory cannot: they share the same $K_0$ group, $\mathbb{Z}/(n-m)\mathbb{Z}$, and  both have trivial $K_1$ (see \cite[\S 5]{ara}). 

Finally, we are able to use the $C^*$-algebras $U_{m,n}^{nc}$ to prove that IBN is preserved under inductive limits. In \cite[Remark, pp1066]{mcclanahan} McClanahan notes that $U_{m,n}^{nc}$ is \emph{semiprojective} in the sense of \cite[\S3]{effros}: that whenever $\{B_i\}$ is an inductive system of $C^*$-algebras with limit $B$ and $\phi:U_{m,n}^{nc}\to B$ is a unital $*$-homomorphism then there exists a unital $*$-homomorphism $\phi_k:U_{m,n}^{nc}\to B_k$ for some $k$.
\begin{proposition} If $\{A_i,\phi_i\}$ is an inductive family of $C^*$-algebras, each with IBN and each $\phi_i$ unital, then the $C^*$-algebraic direct limit $A$ of the system has IBN.
\end{proposition}
\begin{proof} If the limit $A$ did not have IBN then it must have some Basis Type $(N,K)$. By the universal property there is a unital $*$-homomorphism $\psi:U_{N,N+K}^{nc}\to A$ and hence also, because of the semiprojectivity, a unital $*$-homomorphism $\psi_n:U_{N,N+K}^{nc}\to A_n$ for some $n$. But, as $A_n$ has IBN, we would then conclude by Proposition \ref{prophomo} that $U_{N,N+K}^{nc}$ has IBN, a clear contradiction.
\end{proof}

\end{section}

\begin{section}{Stronger Notions}\label{otherIBN}
In \cite{cohn} Cohn considered two ring-theoretic properties strictly stronger than Invariant Basis Number. The $C^*$-algebraic analogues are formulated below.
\begin{definition} A $C^*$-algebra has $\textup{IBN}_1$ if, whenever $n,m$ are integers and $X$ an $A$-module, $A^n\simeq A^m\oplus X$ implies $n\geq m$.
\end{definition}
\begin{definition} A $C^*$-algebra $A$ has $\textup{IBN}_2$ if for all $n>0$, $A^n\simeq A^n\oplus X$ for some $A$-module $X$ implies $X=0$.
\end{definition}
The next proposition is nearly immediate.
\begin{proposition} $\textup{IBN}_2\ \Rightarrow\ \textup{IBN}_1\ \Rightarrow\ \textup{IBN}$.
\end{proposition}
\begin{proof} Suppose $A$ has $\textup{IBN}_2$. If $n<m$ and $A^n\simeq A^m\oplus X$ for some $A$-module $X$ then $A^n\simeq A^n\oplus A^{m-n}\oplus X$ and we conclude by $\textup{IBN}_2$ that $A^{m-n}\oplus X=0$, i.e.\ $m-n=0$ a contradiction.
Suppose that $A$ has $\textup{IBN}_1$. If $A^n\simeq A^m$ for $n>m$ then $A^m\simeq A^n\oplus 0$ and so $n\leq m$, a contradiction.
\end{proof}
Our main goal for this section is twofold: first, to demonstrate that these properties are distinct; and second, to better characterize $C^*$-algebras satisfying the properties $\textup{IBN}_1$ and $\textup{IBN}_2$. This goal is easily accomplished for the property $\textup{IBN}_2$.
\begin{theorem} A $C^*$-algebra $A$ has $\textup{IBN}_2$ if and only if $A$ is stably finite.
\end{theorem}
\begin{proof} Suppose that $A$ is not stably finite, i.e.\ there is a proper isometry $V\in M_n(A)$ for some $n\geq 1$. Note that $I_n\sim VV^*$ and $I_n\sim I_n-VV^*\oplus VV^*\sim I_n- VV^*\oplus I_n$. Thus $A^n\simeq A^n\oplus(I-VV^*)A^n$ where $(I_n-VV^*)A^n\not=0$ as $V$ is proper. Thus $A$ does not have $\textup{IBN}_2$.

Suppose that $A$ does not have $\textup{IBN}_2$. Then $A^n\simeq A^n\oplus X$ for some $n\geq 1$ and nontrivial $A$-module $X$. Note that the embedding $\iota:A^n\to A^n\oplus X$ is an adjointable $A$-module homomorphism which is isometric in the sense that $\iota^*\iota=I_n$. Let $U\in L(A^n\oplus X,A^n)$ be a unitary, then $V=U\circ \iota:A^n\to A^n$ is an adjointable $A$-module homomorphism with $V^*V=I_n$ and $VV^*=U(I_n\oplus 0)U^*\not=I_n$. Thus $V$ corresponds to a $n\times n$ proper matrix isometry and $M_n(A)$ is not finite.
\end{proof}
Since there are $C^*$-algebras with IBN which are not stably finite (for example, the Toeplitz algebra) we conclude that $\textup{IBN}_2$ is strictly stronger than IBN.

Although we do not yet know of a better characterization for $C^*$-algebras with $\textup{IBN}_1$, we are nevertheless able to conclude that it is a distinct property from IBN.

\begin{example} Consider the $C^*$-algebra $\mathcal{T}_2$ which is the universal algebra for two isometries $v_1$ and $v_2$ satisfying $v_1^*v_2=v_2^*v_1=0$ and $v_1v_1^*+v_2v_2^*<1$. Note that $V=[v_1\ v_2]\in M_{1,2}(\mathcal{T}_2)$ is a proper matrix isometry in the sense that $V^*V=I_2$ and $VV^*<1$. Since $V$ is adjointable the submodule $V\mathcal{T}^2_2\subset \mathcal{T}_2$ is complementable (with complement $\ker V^*$) and so
$$\mathcal{T}_2=V\mathcal{T}^2_2\oplus \ker V^*\simeq \mathcal{T}^2_2\oplus \ker V^*.$$
Thus $\mathcal{T}_2$ does not have $\textup{IBN}_1$ but Cuntz \cite[Proposition 3.9]{cuntzk} has shown $K_0(\mathcal{T}_2)=\mathbb{Z}$ and is generated by $[1]_0$, hence $\mathcal{T}_2$ does have IBN.
\end{example}

Indeed, the relationship $A\simeq A^2\oplus X$ guarantees a unital $*$-homomorphism $\phi:\mathcal{T}_2^2\to A$ in much the same way the relationship $A\simeq A^2$ guarantees an embedding $\psi:\mathcal{O}_2\to A$.

\end{section}
\begin{acks} I wish to thank my Ph.D.\ advisor Dr.\ David Pitts for his constant support and insightful questions; my colleague Dr.\ Adam Fuller for our constructive conversations; Dr.\ N.\ Christopher Phillips at the University of Oregon for directing my attention to R{\o}rdam's work, suggesting Proposition \ref{morita}, and remarking that isomorphic standard modules are necessarily unitarily equivalent; and lastly the referee for his or her helpful report. The results present in this paper formed part of my doctoral dissertation while at the University of Nebraska.
\end{acks}

\bibliographystyle{amsplain}
\bibliography{bibl}

\providecommand{\bysame}{\leavevmode\hbox to3em{\hrulefill}\thinspace}
\providecommand{\MR}{\relax\ifhmode\unskip\space\fi MR }
\providecommand{\MRhref}[2]{%
  \href{http://www.ams.org/mathscinet-getitem?mr=#1}{#2}
}
\providecommand{\href}[2]{#2}
\begin{thebibliography}{10}

\bibitem{exel2}
P.~Ara and R.~Exel, \emph{Dynamical systems associated to separated graphs,
  graph algebras, and paradoxical decompositions}, Adv. Math. \textbf{252}
  (2014), 748--804.

\bibitem{exel}
P.~Ara, R.~Exel, and T.~Katsura, \emph{Dynamical systems of type {$(m,n)$} and
  their {$\rm C^*$}-algebras}, Ergodic Theory Dynam. Systems \textbf{33}
  (2013), no.~5, 1291--1325.

\bibitem{ara}
P.~Ara and K.~R. Goodearl, \emph{{$C^\ast$}-algebras of separated graphs}, J.
  Funct. Anal. \textbf{261} (2011), no.~9, 2540--2568.

\bibitem{cohn}
P.~M. Cohn, \emph{Some remarks on the invariant basis property}, Topology
  \textbf{5} (1966), 215--228.

\bibitem{cuntzk}
J.~Cuntz, \emph{{$K$}-theory for certain {$C^{\ast} $}-algebras}, Ann. of Math.
  (2) \textbf{113} (1981), no.~1, 181--197.

\bibitem{effros}
E.~G. Effros and J.~Kaminker, \emph{Homotopy continuity and shape theory for
  {$C^\ast$}-algebras}, Geometric methods in operator algebras ({K}yoto, 1983),
  Pitman Res. Notes Math. Ser., vol. 123, Longman Sci. Tech., Harlow, 1986,
  pp.~152--180.

\bibitem{leavitt1}
W.~G. Leavitt, \emph{Modules over rings of words}, Proc. Amer. Math. Soc.
  \textbf{7} (1956), 188--193.

\bibitem{leavitt3}
\bysame, \emph{Modules without invariant basis number}, Proc. Amer. Math. Soc.
  \textbf{8} (1957), 322--328.

\bibitem{leavitt4}
\bysame, \emph{The module type of a ring}, Trans. Amer. Math. Soc. \textbf{103}
  (1962), 113--130.

\bibitem{mcclanahan}
K.~McClanahan, \emph{{$K$}-theory and {${\rm Ext}$}-theory for rectangular
  unitary {$C^*$}-algebras}, Rocky Mountain J. Math. \textbf{23} (1993), no.~3,
  1063--1080.

\bibitem{raeburn}
I.~Raeburn and D.~Williams, \emph{Morita equivalence and continuous-trace
  {$C^*$}-algebras}, Mathematical Surveys and Monographs, vol.~60, American
  Mathematical Society, Providence, RI, 1998.

\bibitem{rordamstability}
M.~R{\o}rdam, \emph{Stability of {$C^*$}-algebras is not a stable property},
  Doc. Math. \textbf{2} (1997), 375--386 (electronic).

\bibitem{rordamsums}
\bysame, \emph{On sums of finite projections}, Operator algebras and operator
  theory ({S}hanghai, 1997), Contemp. Math., vol. 228, Amer. Math. Soc.,
  Providence, RI, 1998, pp.~327--340.

\bibitem{rordambook}
M.~R{\o}rdam, F.~Larsen, and N.~Laustsen, \emph{An introduction to {$K$}-theory
  for {$C^*$}-algebras}, London Mathematical Society Student Texts, vol.~49,
  Cambridge University Press, Cambridge, 2000.

\bibitem{rordamclass}
M.~R{\o}rdam and E.~St{\o}rmer, \emph{Classification of nuclear
  {$C^*$}-algebras. {E}ntropy in operator algebras}, Encyclopaedia of
  Mathematical Sciences, vol. 126, Springer-Verlag, Berlin, 2002, Operator
  Algebras and Non-commutative Geometry, 7.

\bibitem{wegge}
N.~E. Wegge-Olsen, \emph{{$K$}-theory and {$C^*$}-algebras}, Oxford Science
  Publications, The Clarendon Press, Oxford University Press, New York, 1993, A
  friendly approach. \MR{1222415 (95c:46116)}

\end{thebibliography}

\end{document}